\documentclass[12pt]{amsart}
\usepackage[cp1250]{inputenc}
\usepackage[T1]{fontenc}
\usepackage{amscd}
\usepackage{amssymb}

\numberwithin{equation}{section}
\def\wyr#1{\textit{#1}}

\def\r{\rightarrow}
\def\t{\times}
\def\s{\subseteq}

\def\a{\alpha}

\def\C{\mathcal{C}}

\def\F{\mathcal{F}}
\def\H{\mathcal{H}}
\def\G{\mathcal{G}}
\def\U{\mathcal{U}}
\def\V{\mathcal{V}}

\def\K{\mathcal{K}}

\def\N{\mathcal{N}}

\def\SC{\mathcal C\mathcal N}

\def\SS{\mathcal{G}}
\def\ZZ{\mathcal{Z}}

\def\t{\times}
\def\ci{\circ}
\def\cd{\cdot}
\def\rz{\Bbb R}

\def\ld{,\ldots ,}

\def\12{{1\over 2}}
\def\rz{\mathbb{R}}
\def\R{\rz}
\def\l{\lambda}
\def\wyr#1{\textit{#1}}
\def\f{\Phi}
\def\xg{\frak X_G}

 \DeclareMathOperator{\Sect}{Sect}
 \DeclareMathOperator{\gau}{Gau}
 \DeclareMathOperator{\conj}{conj}
 \DeclareMathOperator{\cl}{cl}
 \DeclareMathOperator{\intt}{int}
\DeclareMathOperator{\id}{id} 
\DeclareMathOperator{\ro}{Ro} \DeclareMathOperator{\supp}{supp}
 \DeclareMathOperator{\dom}{dom}
 
\def\rox{\ro(X)}
\def\roy{\ro(Y)}
\def\gg{\frak g}
\def\p{\partial}
\def\la{\langle}
\def\ra{\rangle}

\DeclareMathOperator{\Aut}{Aut} 
  \DeclareMathOperator{\rob}{Rob}
  
 \DeclareMathOperator{\var}{var}
\DeclareMathOperator{\carr}{carr} 
 \DeclareMathOperator{\pr}{pr}
\DeclareMathOperator{\ev}{ev}

 \theoremstyle{plain}
\newtheorem{thm}{Theorem}[section]
\newtheorem{lem}[thm]{Lemma}
\newtheorem{prop}[thm]{Proposition}
\newtheorem{cor}[thm]{Corollary}

\theoremstyle{definition}
\newtheorem{dff}[thm]{Definition}
\newtheorem{rem}[thm]{Remark}

\thanks{The second author is partially supported by the Polish Ministry of Science and Higher Education and the
AGH grant n. 11.420.04}
\address{(M.R.) Department of Mathematics, Ben Gourion University \linebreak of the Negev, Beer Sheva,
Israel}  \email {matti@math.bgu.ac.il}
\address{ (T.R.) Faculty of
Applied Mathematics, AGH University of Science  \linebreak  and
Technology, al. Mickiewicza 30, 30-059 Krak\'ow, Poland}
\email{tomasz@uci.agh.edu.pl}

\title{Isomorphisms between groups of equivariant homeomorphisms of  $G$-manifolds with one orbit type}
\author{Matatyahu Rubin, Tomasz Rybicki}

\begin{document}

\maketitle

\begin{abstract}
Given a compact Lie group  $G$, a reconstruction theorem for free
$G$-manifolds is proved. As a by-product reconstruction results
for locally trivial bundles are presented.  Next,  the main
theorem is generalized to $G$-manifolds with one orbit type. These
are the first reconstruction results in the category of
$G$-spaces, showing also that the reconstruction in this category
is very specific and involved.
\end{abstract}

\section{Introduction}
Let $X$ be a regular ($T_3$) topological space and let $\H(X)$ be
the  group of all  homeomorphisms of $X$. If $H\leq\H(X)$ then the
pair $\la X,H\ra$ is called a \wyr{space-group pair}. A class $\K$
of space-group pairs is called a \wyr{faithful class}, if for
every $\la X_1,H_1\ra,\la X_2,H_2\ra\in \K$ and a group
isomorphism $\varphi:H_1\cong H_2$, there exists a homeomorphism
$\tau:X_1\cong X_2$ such that $\varphi (h)=\tau \circ h \circ
\tau^{-1}$ for every $h \in H_1$. The reconstruction problem in
the topological category consists in finding out faithful
space-group pairs. A pioneer work was done here by Whittaker
\cite{whit}. Such problems are also considered in many other
categories, see \cite{Ru1}, \cite{Ru2}, \cite{Ru3}, \cite{ban1},
\cite{RY}, \cite{Ryb1},  \cite{bo-br}.

There are several reconstruction  results  on the ground of
differential geometry. First of all it is well known (\cite{whit},
\cite{fil}, \cite{Ru1}, \cite{ban-m}) that the~group of~all
$C^r$-diffeomorphisms ($0\leq r\leq \infty$) of~a~$C^r$-manifold
defines uniquely the~topological and~smooth structure of~the
manifold. Analogous results are true for the automorphism groups
of some geo\-metric structures, e.g. \cite{ban-m},
 \cite{Ryb4}, \cite{Ryb5}, \cite{Ryb2},  \cite{kmr},
\cite{be-ru}, \cite{Ru4}. Infinitesimal counterparts of the
reconstruction theorems are also known and useful in the proofs of
them (see, e.g., \cite{om}, \cite{gr}, \cite{gr-po} and references
therein).

From now on we denote by $H\leq K$ (resp. $H\lhd K$) the fact that
$H$ is a subgroup (resp. normal subgroup) of a group $K$.

 A group of
homeomorphisms $H\leq\H(X)$ is called \wyr{factorizable} if for
every open covering $\U$ of the space $X$, the set
$\bigcup_{U\in\U}H\begin{tabular}{|@{} c @{}|}
   $U$\\
  \hline
\end{tabular}$
generates $H$. Here $H\begin{tabular}{|@{} c @{}|}
   $U$\\
  \hline
\end{tabular}=\{h\in H:\, h|_{X\setminus U}=\id\}$.

A group $H\leq\H(X)$ is said to be \emph{non-fixing} if $H(x)\neq
\{ x \}$ for every $x \in X$, where $H(x):= \{ h(x)|h \in H \}$ is
the orbit of $H$ at $x$.

The following theorem, due to  the first-named author
(\cite{Ru1}, \cite{Ru3}), is a basic fact in reconstruction
problems of homeomorphism groups.

\begin{thm}
Let $X_1, X_2$ be regular topological spaces  and let $H_1$ and
$H_2$ be factorizable, non-fixing homeomorphism groups of $X_1$
and $X_2$, resp. Suppose that there is an isomorphism
$\varphi:H_1\cong H_2$. Then there is a~homeomorphism $\tau :X_1
\cong X_2$ such that $\varphi (h)=\tau \circ h \circ \tau^{-1}$
for every $h \in H_1$.
\end{thm}

Let $G$ be a compact  Lie group  acting on a topological space
$X$. A homeomorphism $f:X\r X$ is $G$-equivariant, if for every
$g\in G$, $x\in X$, $f(g.x)=g.f(x)$. That is, $\mu_g\ci f=f\ci
\mu_g,\,\forall g\in G$, where $\mu_g:X\ni x\mapsto g.x \in X$ is
the left translation. The symbol  $\H_G(X)_0$ will stand for the
group of all  homeomorphisms of $X$ that can be joined with the
identity by a compactly supported isotopy consisting of
$G$-equivariant homeomorphisms. In general, $\H(X)_0$ will denote
the group of all  homeomorphisms of $X$ that can be joined with
the identity by a compactly supported isotopy.

Let $M$ be a topological manifold with a free $G$-action. Our aim
is to prove a reconstruction theorem for  $\H_G(M)_0$. The case of
free $G$-action can be viewed as a generic case among $G$-actions,
where $G$ is a compact Lie group (\cite{br}).

Recall the following basic fact, c.f. \cite{gle}, \cite{br}. If
$G$ is a compact  Lie group and $G$ acts freely on a
$T_{3\frac{1}{2}}$-space $X$ then
 $X$ can be regarded as the total space of a principal $G$-bundle
$\pi:X\r B_X$. Let $\tilde\pi:\H_G(X)_0\ni f\mapsto\tilde
f\in\H(B_X)_0$ be the induced homomorphism. Next, let
$\gau(X):=\ker(\tilde\pi)$ be the group of gauge transformations
on the total space $X$ of $\pi:X\r B_X$.

By $\Sect(X)$ we denote the set of all continuous sections of the
principal $G$-bundle $\pi:X\r B_X$. If $U$ is an open set of $B_X$
such that $\pi$ is trivial over $U$ then $\Sect(\pi^{-1}(U))$
admits a group structure by pointwise multiplication.

Our main result is the following
\begin{thm}
Let $G$ be a compact  Lie group acting freely  on paracompact
connected smooth manifolds $M$ and $N$.
 If there is a  group isomorphism
$\f:\H_G(M)_0\cong \H_G(N)_0$ then
\begin{enumerate}
\item there is a homeomorphism $\tau:B_M\r B_N$ such that
$\widetilde{\f(f)}=\tau\ci \tilde f\ci\tau^{-1}$ for all
$f\in\H_G(M)_0$;\item there is a continuous mapping
$\bar\f:B_M\r\Aut(G)$; and \item for all $x\in B_M$ there are an
open neighborhood $U$ of $x$ in $B_M$,  isomorphisms
$\phi_U:\Sect(\pi^{-1}(U))\cong \gau(\pi^{-1}(U))$,
$\phi_{\tau(U)}:\Sect(\pi^{-1}(\tau(U)))\cong
\gau(\pi^{-1}(\tau(U)))$
 and a homeomorphism
 $\sigma_{U}:\pi^{-1}(U)\cong\pi^{-1}(\tau(U))$ induced by
 $\bar\f$ such that
$$\f(h)=(\phi_{\tau(U)}\ci\hat\sigma_U\ci\phi_{U}^{-1})(h)$$
for every $h\in \gau(\pi^{-1}(U))$, where
$\hat\sigma_U(f)=\sigma_U\ci f\ci\tau^{-1}$.
\end{enumerate}
\end{thm}

If we consider globally trivial principal bundles the theorem
assumes a simpler form, namely we have

\begin{cor} Let $G$ be a compact Lie group and
let $M=B_M\t G$ and $N=B_N\t G$ be free product paracompact
connected smooth $G$-manifolds.
 If there is a  group isomorphism
$\f:\H_G(M)_0\cong \H_G(N)_0$ then
\begin{enumerate}
\item there is a homeomorphism $\tau:B_M\r B_N$ such that
$\widetilde{\f(f)}=\tau\ci \tilde f\ci\tau^{-1}$ for all
$f\in\H_G(M)_0$; and \item there exists
  a  homeomorphism $\sigma:M\cong N$
such that $\f(h)=\sigma\ci h\ci\tau^{-1}$ for all $h\in\gau(M)$
(so $\sigma$ is fiberwise over $\tau$).
\end{enumerate}
\end{cor}
The proof follows from that of Theorem 1.2.

Notice that in \cite{Ryb3} it was proved by the second-named
author that $\H_G(M)_0$ is perfect, i.e. equal to its commutator
subgroup, provided the $G$-action is of one orbit type. It is very
likely that if we replace $\H_G(M)_0$ by $\H_G(M)$, the group of
all $G$-equivariant homeomorphisms of $M$, such a theorem would be
false. The fact that $\H_G(M)_0$ is perfect occurs to be an
indispensable ingredient of the proof of Theorem 1.2 (see
Proposition 3.11).

Section 2 is devoted to homeomorphism groups related to locally
trivial bundles. It contains some reconstruction results being
consequences of Theorem 1.1. These results are applied in the
proof of Theorem 1.2 but they are also interesting for themselves.
The next section 3 is a clue part of the paper. It contains a
characterization of transversal isotropy subgroups $S_x^M$ and
isotropy subgroups $F_x^M$, $x\in B_M$, of $\H_G(M)_0$. The
characterization of $F_x^M$ is very delicate and, though
$F_x^M\leq \gau(M)$, it is not valid for $F_x^M$ as subgroups of
the gauge group $\gau(M)$. A demanding  problem is whether it is
possible to obtain a reconstruction result also from the gauge
group $\gau(M)$.  Another version of Theorem 1.2 is formulated in
section 4. In section 5 we generalize Theorem 1.2 to the case of
$G$-action with one orbit type.

It occurs that the proof of Theorem 1.2 cannot be carried over to
the $C^r$ category, $r=1\ld\infty$, without possible essential
changes in whole strategy of the proof. Namely, Lemma 3.8(2) is
obviously no longer true for homeomorphisms of class $C^r$. We can
neither drop the assumption that $M$ and $N$ are topological
manifolds due to the proofs of Lemma 3.9 and Proposition 3.10.

\bigskip

\emph{Acknowledgements}. A part of this work was written when the
second-named author visited Ben Gourion University in January
2010. He thanks for the invitation and hospitality.

\section{Homeomorphism groups related to locally trivial bundles  }
Let $X$ be a regular topological space. For $h\in\H(X)$ we set
$\carr(h):=\{x\in X:\, h(x)\neq x\}$, $\supp(h):=\cl(\carr(h))$
and $\var(h):=\intt(\cl(\carr(h)))=\intt(\supp(h))$. Here, for
$A\s X$, $\cl(A)$ and $\intt(A)$ denote the closure and the
interior of $A$, respectively.

Let $H\leq\H(X)$ be a group of homeomorphisms on $X$. For any open
$U\s X$, denote $H_U=\{h\in H:\,\supp(h)\s U\}$. In particular,
$\H_U(X)=\{h\in \H(X):\,\supp(h)\s U\}$. If $X$ is a paracompact
space, then  $H$ is factorizable in the above sense iff
$\bigcup_{U\in\U}H_U$ generates $H$ for any open cover $\U$ of
$X$.

Recall that an open $U\s X$ is called a \wyr{regular open set} if
$U=\intt(\cl(U))$. Notice  that for any set $A\s X$, the set
$\intt(\cl(A))$ is regular open. It follows from definition that
for any regular open set $U$ and any $h\in\H(X)$, $\carr(h)\s U$
iff $\var(h)\s U$.

Let $\ro(X)$ denote the family of regular open sets of the space
$X$. We endow $\ro(X)$ with the following operations:
$U+V:=\intt(\cl(U\cup V))$, $U\cdot V:=U\cap V$, and
$-U:=\intt(X\setminus U)$. Then $\la\ro(X), +, \cdot, -\ra$ is a
complete Boolean algebra. Clearly, $0^{\ro(X)}=\emptyset$,
$1^{\ro(X)}=X$, and the induced partial ordering of $\ro(X)$ is
$\leq^{\rox}=\s$.

Let $f$ be a homeomorphism between $X$ and $Y$. Then $f$ induces
an isomorphism $f^{\ro}$ between $\rox$ and $\roy$ given by
$f^{\ro}(U)=f(U):=\{f(x)| x\in U\}$. If $X$ is Hausdorff then the
mapping $f\mapsto f^{\ro}$ is an embedding of $H(X)$ into
$\Aut(\rox)$.

 A space-group pair $\la X,H\ra$ is called
a \wyr{local movement system} if for every $U\in\rox$ there is
$f\in H$, $f\neq\id$, such that $\var(f)\s U$. $H$ is then called
a \wyr{locally moving} group of $X$. A starting point in many
reconstruction problems is the following theorem proved by  Rubin
\cite{Ru1}, \cite{Ru3}.

\begin{thm}
Let $\la X_1,H_1\ra$ and $\la X_2,H_2\ra$ be local movement
systems, and $\varphi: H_1\cong H_2$. Then there exists a unique
$\psi:\ro(X_1)\cong\ro(X_2)$ such that $\varphi(f)^{\ro}=\psi\ci
f^{\ro}\ci \psi^{-1}$ for every $f\in H_1$. In other words, for
every $U,V\in\rox$ and $f\in H_1$,
$$
V=f(U)\quad\Leftrightarrow\quad \psi(V)=\varphi(f)(\psi(U)).
$$
\end{thm}

Now, we wish to generalize Theorem 2.1 to the case of locally
trivial bundles.

Let $F$ be a topological space. Recall that a continuous
surjective mapping $\pi:X\r B$ is called a \wyr {locally trivial
bundle with the standard fiber $F$} if the following local
triviality property holds:

There is an open cover $\U$ of $B$ such that for every $U\in\U$
there is a homeomorphism $\a_U: \pi^{-1}(U)\cong U\t F$  such that
$\pr_1\ci\a_U=\pi|_{\pi^{-1}(U)}$ with $\pr_1:U\t F\r U$ being the
canonical projection. We then say that $\pi$ is trivial over $U$.
It follows that for $x\in U$ the map
$\a_U|_{\pi^{-1}(x)}:\pi^{-1}(x)\r\{x\}\t F$ is a homeomorphism.
Observe that $\pi$ is always trivial over a contractible set.

\begin{prop} Let $\pi:X\r B$ be a locally trivial bundle as above
and let $A$ be a subset of $B$. Then
$$\pi^{-1}(\cl(A))=\cl(\pi^{-1}(A))\quad\hbox{ and}\quad
\pi^{-1}(\intt(A))=\intt(\pi^{-1}(A)).$$
 Consequently, if $U\in\ro(B)$
then $\pi^{-1}(U)\in\ro(X)$.
\end{prop}

The proof is an easy exercise. Notice that the local triviality
property is an indispensable assumption in Proposition 2.2.

Suppose now that two locally trivial bundles $\pi_1:X_1\r B_1$ and
$\pi_2:X_2\r B_2$ with the same standard fiber $F$ are given. A
homeomorphism $f:X_1\r X_2$ is said to be \wyr{projectable} if
there exists a homeomorphism $\tilde f:B_1\r B_2$ such that
$\pi_2\ci f=\tilde f\ci\pi_1$. It follows that such a $\tilde f$
is then uniquely determined.

Now, if a locally trivial bundle $\pi:X\r B$ is given, we define
$\H^{proj}(X):=\{f\in\H(X):\, f\,\hbox{is projectable}\}$, the
group of projectable homeomorphisms of $X$. Denote by $\tilde
\pi:\H^{proj}(X)\r\H(B)$ the mapping induced by $\pi$, i.e.
$\tilde\pi(f)=\tilde f$. We say that a group of homeomorphisms
$H(X)$ of $X$ is \wyr{projectable} if $H(X)\leq\H^{proj}(X)$.

\medskip
{\it Convention}. Under the above, let $H(X)$ be a projectable
group. For any $f\in H(X)$, $\tilde f$ will always denote an
element of $\H(B)$ given by $\tilde f:=\tilde{\pi}(f)$, where
$\tilde\pi:H(X)\r\H(B)$ is the mapping induced by $\pi$. Likewise,
we denote $\tilde{H}(B):=\tilde\pi(H(X))$.
\medskip

We are in a position to formulate a straightforward consequence of
Theorem 2.1. We say that a projectable group $H(X)$ is
\wyr{transversally locally moving}, if for every $U\in\ro(B)$
there is $f\in H(X)$ such that $\var(\tilde f)\s U$ and $\tilde
f\neq\id$.

\begin{thm}
 Let $\pi_1:X_1\r B_1$ and $\pi_2:X_2\r B_2$ be two locally
trivial bundles with the same standard fiber $F$. Assume that
$H(X_i)$, $i=1,2$, are projectable, transversally locally moving
groups, and $\varphi: H(X_1)\cong H(X_2)$. Then there exists a
unique $\psi:\ro(B_1)\cong\ro(B_2)$ such that
$\widetilde{\varphi(f)}^{\ro}=\psi\ci \tilde f^{\ro}\ci \psi^{-1}$
for every $f\in H(X_1)$. In other words, for every
$U,V\in\ro(B_1)$ and $f\in H(X_1)$,
$$
V=\tilde f(U)\quad\Leftrightarrow\quad
\psi(V)=\widetilde{\varphi(f)}(\psi(U)).
$$
\end{thm}
Observe that it is not known whether, under the assumptions of
Theorem 2.3, the Boolean algebras $\ro(X_1)$ and $\ro(X_2)$ are
isomorphic.

For any projectable group $H(X)$ we have that $\ker
\tilde\pi=\{f\in H(X):\, \tilde f=\id\}$ is a normal subgroup of
 $H(X)$. Now, if $\varphi: H(X_1)\cong
H(X_2)$ then $\varphi:\ker \tilde\pi_1\cong\ker \tilde\pi_2$. If
$\tilde{H}(B):=\tilde\pi(H(X))$, then $\tilde{H}(B)\leq\H(B)$ and
$H(X)/\ker \tilde\pi\cong\tilde{H}(B)$. As usual, for
$U\in\ro(B)$,  set $\tilde H_U(B):=\{h\in\tilde H(B):\,\var(h)\s
U\}$.

\begin{dff}
Let $H(X)$ be a projectable group of a locally trivial bundle
$\pi: X\r B$. Then $H(X)$ is called \wyr{transversely
factorizable} (resp. \wyr{transversely non-fixing}; resp.
\wyr{transversely transitive}; resp. \wyr{transversely LDC}) if
$\tilde H(B)$ is factorizable (resp. non-fixing; resp. transitive;
resp. LDC). Recall that a homeomorphism group $H(X)$ is
\wyr{locally densely conjugated} or \wyr{LDC} for short (c.f.,
\cite{be-ru}) if for any $U\in\rox$
and $x\in U$ the orbit $H\begin{tabular}{|@{} c @{}|} $U$\\
\hline\end{tabular}(x)$ is somewhere dense.

\end{dff}

\begin{thm}
Let $\pi_1:X_1\r B_1$ and $\pi_2:X_2\r B_2$ be two locally trivial
bundles with  standard fiber $F$ and $\varphi: H(X_1)\cong
H(X_2)$. Suppose that $H(X_i)$, $i=1,2$, are projectable and
fulfil one of the following conditions $(i=1,2)$:
\begin{enumerate}
\item $H(X_i)$ are transversely factorizable and transversely
non-fixing; \item There are $H_i\leq \tilde H(B_i)$ which are
factorizable and non-fixing, and for any $x_i\in B_i$, the orbit
$\tilde H(B_i)(x_i)$ is somewhere dense; \item $H(X_i)$ are
transversely LDC groups.
\end{enumerate}
Then there is a homeomorphism $\tau:B_1\cong B_2$ such that for
every $f\in H(X_1)$, $\,\widetilde{\f(f)}=\tau\ci\tilde
f\ci\tau^{-1}$.
\end{thm}

The proof follows immediately from some results in \cite{be-ru}.

\section{Characterizations of isotropy subgroups}

Let $H\leq K$. For any $g\in K$, $\ZZ^H(g):=\{h\in H|gh=hg\}$ is
the centralizer of $h$ in $H$, and $\ZZ^K(H)$ denotes the
centralizer of $H$ in of $K$. In particular, $\ZZ(K)$ is the
center of $K$. Next, $\N^K(H)$ stands for the normalizer of $H$ in
$K$.

Let us start with the following structural theorem, \cite{gle}.
Observe that originally this theorem was proved for
$T_{3\frac{1}{2}}$-spaces.

\begin{thm}
Let $M$ be a connected paracompact manifold, let $G$ be a compact
 Lie group acting freely on $M$. The space of orbits, $B_M$, is a
connected manifold and the orbit map $\pi:M\r B_M$ is a principal
$G$-bundle, where the structure group $G$ acts by right
translations on fibers. Conversely, every principal $G$-bundle is
obtained from such an action.
\end{thm}

Let $M$ be a smooth manifold with $G$-action. A vector field  $X$
on $M$ is called $G$-invariant vector field, if $T\mu_g\ci
X=X\ci\mu_g$ for all $g\in G$, where $\mu_g$ is the left
translation. The symbol $\xg(M)$ will stand for the Lie algebra of
all $G$-invariant vector fields. The following result is
well-known.

\begin{lem}
For any coordinate chart $(U;(x_1\ld x_n))$ of $B_M$ such that
$\pi$ is trivial over $U$, any $ X\in\xg(M)$ assumes the form
\begin{equation*}
X=\sum_{i=1}^nu_i(x_1\ld x_n)\frac{\p}{\p
x_i}+\sum_{j=1}^{m}v_j(x_1\ld x_n)Y_j
\end{equation*}
on $\pi^{-1}(U)$, where $u_i$, $v_j$ are $C^{\infty}$-functions on
$U$ and $(Y_1\ld Y_m)$ is a basis of the Lie algebra $\gg$ of $G$.
\end{lem}
\begin{cor}
Let $M$ be as in Theorem 3.1. Then $\H_G(M)_0$ satisfies
conditions (1), (2) and (3) in Theorem 2.5. Moreover, $\H_G(M)_0$
is transversely transitive.
\end{cor}
\begin{proof} (1) is a consequence of the fragmentation property for homeomorphisms,
c.f. \cite{ed-ki}.  If $X\in\xg(M)$ with $\supp(X)\s U$ then the
flow of $X$ consists of elements of $\H_G(M)$. Now to show (2) it
suffices to take the group generated by elements of flows of the
above form. This group is also transversally transitive, hence
(3).
\end{proof}

\wyr{Notation}. Let $M$ be a manifold. The symbol $\rob(M)$ stands
for the set of all regular open subsets $U$ of $M$ such that $U$
is an embedded ball. Next, for $x\in M$ denote by $\rob_x(M)$ the
subfamily of $\rob(M)$ of all neighborhoods of $x$.
\medskip

Let $(M,B_M,\pi, G)$ be a principal $G$-bundle over $B_M$. Then
$M$ is uniquely determined by its cocycle of transition functions,
that is a covering $\{U_i\}_{i\in I}$ of $B_M$ by coordinate chart
domains and a collection of mappings $g_{ji}:U_i\cap U_j\r G$ such
that $$\forall x\in U_i\cap U_j \cap U_k,\quad g_{kj}(x).
g_{ji}(x)=g_{ki}(x).$$ Then the maps $(x,g)\mapsto (x,g_{ji}(x).
g)$, $x\in U_i\cap U_j$, $g\in G$, are the transition functions of
$(M,B_M,\pi, G)$  related to $\{U_i\}_{i\in I}$. In particular, we
are given a bundle atlas $\{\phi_i:\pi^{-1}(U_i)\r U_i\t G\}$ such
that we get $(\phi_i\ci\phi_j^{-1})(x,g)=(x,g_{ij}(x).g)$.

Assume that we have two cocycles $g_{ij}$ and $g'_{ij}$ over the
same covering $\{U_i\}_{i\in I}$ of $B_M$ (otherwise we could pass
to a common refinement). Then $g_{ij}$ and $g'_{ij}$ are called
\wyr{cohomologous} if there is a family $\chi_i:U_i\r G$, $i\in
I$, such that $\chi_i(x).g_{ij}(x)=g'_{ij}(x).\chi_j(x)$ for all
$x\in U_{ij}$. Clearly, the principal bundle is uniquely defined
by the cohomology class of its cocycle of  transition functions.

Every principal bundle admits a unique right action $r:P\t G\r P$,
called the \wyr{principal right action}, given by
$\phi_i(r(\phi_i^{-1}(x,g),h))=(x,gh)$. Notice that this is well
defined, since the left and right translation on $G$ commute.


We have the following version of Theorem 2.5.

\begin{thm}
Let  $\f: \H_G(M)_0\cong \H_G(N)_0$. Then there exists a unique
$\psi:\ro(B_M)\cong\ro(B_N)$ such that
$\widetilde{\f(f)}^{\ro}=\psi\ci \tilde f^{\ro}\ci \psi^{-1}$ for
every $f\in \H_G(M)_0$. In other words, for every $U,V\in\ro(B_M)$
and $f\in \H_G(M)_0$,
$$
V=\tilde f(U)\quad\Leftrightarrow\quad
\psi(V)=\widetilde{\f(f)}(\psi(U)).
$$
Moreover, there exists a homeomorphism $\tau:B_M\cong B_N$ such
that for every $f\in \H_G(M)_0$,
$\,\widetilde{\f(f)}=\tau\ci\tilde f\ci\tau^{-1}$.
\end{thm}
In fact, the groups $\H_G(M)_0$ and $\H_G(N)_0$ are obviously
projectable and fulfill conditions (1), (2) and (3) from Theorem
2.5.

 For $x\in B_M$ we denote by
$$S_x^M:=\{f\in \H_G(M)_0| \tilde f(x)=x\},$$  the \wyr{transversal
isotropy subgroup} of $\H_G(M)_0$ at $x$. Under the notation of
Theorem 3.4, it is clear that $\forall x\in B_M$,
$\f(S_x^M)=S_{\tau(x)}^N$. Next, for $U\in\ro(B_M)$, denote
$$S_U^M:=\{f\in \H_G(M)_0:\, \var(\tilde f)\s U\}.$$ Then
$\f(S_U^M)\s S_{\psi(U)}^N$.

Let $$\gau(M):=\{f\in \H_G(M)_0|\, \tilde f=\id_{B_M}\}$$ be the
subgroup of $\H_G(M)_0$ of all its  \wyr{gauge transformations} of
the bundle $M$. (For the significance of this notion in
mathematics and physics, see, e.g., \cite{mich}). Clearly
$\gau(M)$ is a normal subgroup of $\H_G(M)$, and
$\f(\gau(M))=\gau(N)$.

For $U\in\ro(B_M)$ denote $$F_U^M:=\{f\in \H_G(M)_0:\, \var(f)\s
\pi^{-1}(U)\}.$$ Analogously, $\gau_U(M):=\{f\in \gau(M):\, \var(
f)\s \pi^{-1}(U)\}$.

Let $\{U_i\}_{i\in I}$ be an open covering by elements of
$\rob(B_M)$. Assume that $\phi_i:\pi^{-1}(U_i)\r U_i\t G$ is a
local trivialization on $M$. If $f\in\gau(M)$  then
$f|_{\pi^{-1}(U_i)}$ is identified with a mapping $ f^{(i)}:U_i\r
G$ as follows
\begin{equation} (\phi_i\ci f\ci\phi_i^{-1})(x,g)=(x, f^{(i)}(x).g)\quad\hbox{for}\quad (x,g)\in U\t
G.\end{equation} In particular, for $g\in G$ we define
\begin{equation}
c_i(g):\pi^{-1}(U_i)\r \pi^{-1}(U_i)\quad\hbox{by}\quad(\phi_i\ci
c_i(g)\ci\phi_i^{-1})(x,h)=(x, gh),
\end{equation}
that is $c_i(g)^{(i)}$ is equal to the constant mapping $g$ on
$U_i$. It is obvious that $c_i(g)\in\gau(\pi^{-1}(U_i))$ and it
depends on the choice of $U_i$.

On the other hand, if $g\in\ZZ(G)$ then $c_i(g)$ is independent of
chart and it extends to $c(g)\in\gau(M)$. Moreover,
$c(g)\in\ZZ(\gau(M))$.

 Next, let $h\in S_{U_i}^M$. For every $(x,g)\in U_i\t G$ we
may write $$(\phi_i\ci h\ci\phi_i^{-1})(x,g)=(h_1(x),
h_2(x,g))=(h_1(x), h^{(i)}_2(x).g),$$ where $h^{(i)}_2:U_i\r G$.
Here $h_1$ can be viewed as an element of $\tilde H(B_M)$, and
$h_2$ as an element of $\gau(\pi^{-1}(U_i))$. Define
\begin{equation}H^{(i)}_{U_i}:=\{h\in F_{U_i}^M:\,
h_2^{(i)}=e\}.\end{equation} Then $H^{(i)}_{U_i}\leq F_{U_i}^M$
and $H^{(i)}_{U_i}\cong \H_{U_i}(B_M)$. This definition depends on
the choice of a chart over $U_i$. Moreover, $F_{U_i}^M\cong
H_{U_i}^{(i)}\rtimes \gau_{U_i}(M)$, a semi-direct product. That
is, if $h=(h_1,h_2), k=(k_1,k_2)\in H_{U_i}^{(i)}\times
\gau_{U_i}(M)$ then $h\ci k=(h_1\ci k_1,( h_2^{(i)}\ci k_1).
k^{(i)}_2)$ on $\pi^{-1}(U_i)$.

\begin{rem} In general, the above construction of $ f^{(i)}$ cannot
occur globally. In fact, let $(M,B_M,\pi, G)$ be a principal
$G$-bundle  and let $\l:G\t S \r S$ be a left action of the
structure group $G$ on a manifold $S$. We consider the right
action $r:(M\t S)\t G\r M\t S$ given by $r((u,s),g)=(u.g,
g^{-1}.s)$. Then $M\t_GS$, the space of orbits of the action $r$,
carries a unique manifold structure, and $\bar\pi: M\t_G S\r B_M$
is a locally trivial bundle with the standard fiber $S$. It is
denoted by $M[S,\l]$ and called the \wyr{associated bundle} for
the action $\l$. In particular, for the conjugation action
$\conj:G\t G\ni (g,h)\mapsto ghg^{-1}\in G$ we get the associated
bundle $M[G,\conj]$.

 Let $\C(M,S)^G$ stand for the space of
all mappings $f:M\r S$ which are $G$-equivariant, i.e.
$f(u.g)=g^{-1}.f(u)$ for $g\in G$ and $u\in M$. Then we have a
bijection between $\C(M,S)^G$ and the space of sections of the
associated bundle $M[S,\l]$, see, e.g., \cite{br}, \cite{KM}.

It is well known that the group $\gau(M)$ coincides with the space
of $G$-equivariant mappings $\C(M,(G,\conj))^G$. It follows that
$\gau(M)$ identifies with $\C(B_M\leftarrow M[G,\conj])$, the
space of sections of the associated bundle $M[G,\conj]$.
Consequently the above mappings $f^{(i)}$  do not extend to $B_M$
in general.

\end{rem}

\begin{dff}\begin{enumerate}
\item Let $f\in \H_G(M)_0$ and $U\in\rob(B_M)$. Then $f$ is said
to be \wyr{constant-like over $U$} if there exists  a subgroup
$\SS_U(f)\leq S_U^M$ such that $\tilde\pi(\SS_U(f))$ is transitive
on $U$ and $\SS_U(f)\s \ZZ^{\H_G(M)_0}(f)$. Let $\C_U$ denote the
set of all constant-like elements of $\H_G(M)_0$ over $U$. \item
$f\in\gau(M)$ is called \wyr{globally constant-like} if there
exists an open covering $\U\s\rob(B_M)$ such that
$f\in\bigcap_{U\in\U}\C_U$. By $\C_{M}$ we denote the set of all
globally constant-like elements of $\gau(M)$. \item Let $x\in
B_M$. $f\in \H_G(M)_0$ is called  \wyr{constant-like near $x$} if
there is $U\in\rob(B_M)$ such that $x\in\cl(U)$ and $f\in\C_U$. By
$\SC_{x}$ we denote the set of all constant-like near $x$ elements
of $\H_G(M)_0$.
\end{enumerate}
\end{dff}

Obviously, for any $U\in\rob(B_M)$ we have $\id_M\in\C_M\s\C_U$.
We also have that $c(g)\in\C_M$ for all $g\in\ZZ(G)$. But there
are other elements of $\C_U$ as the following shows.

\begin{prop}
\begin{enumerate}
\item Every $c_i(g)$ extends to an element $\hat c_i(g)$ of
$\gau(M)$ which is constant-like over $U_i$. \item
$\ZZ^{\H_G(M)_0}(\gau(M))=\{c(g)|\, g\in \ZZ(G)\}$. \item For
every $U\in\rob(B_M)$ one has
$\ZZ^{\H_G(M)_0}(\gau(M))\cd\C_U=\C_U$.
\end{enumerate}
\end{prop}
\begin{proof}
(1) Since $U_i$ is an embedded ball in $B_M$ we may assume that
there are $V_i, W_i\in\rob(B_M)$ such that $\cl(U_i)\s
V_i\s\cl(V_i)\s W_i$. By using that $G$ is a Lie group, a standard
argument using a chart in $G$ (c.f. the proof of 3.9(1) below)
leads to the existence of $\hat c_i(g)^{(i)}:W_i\r G$ extending
$c_i(g)^{(i)}:U_i\r G$ such that $\hat c_i(g)^{(i)}=e$ off $V_i$.
Clearly, $\hat c_i(g)^{(i)}$ corresponds to an element $\hat
c_i(g)\in\gau(M)$ which is constant-like over $U_i$ with
$\SS_{U_i}(\hat c_i(g))=H^{(i)}_{U_i}$ given by (3.3).

(2) The inclusion $\supseteq$ is trivial. Now let $f\neq c(g)$ for
all $g\in\ZZ(G)$. If there are some $i\in I$ and $x\in U_i$ such
that $f^{(i)}(x)\not\in\ZZ(G)$ then we are done. Otherwise, for
some $i\in I$, $f^{(i)}$ is not constant. Then there is $h\in
H^{(i)}_{U_i}$ such that $hf\neq fh$. Therefore, the inclusion
$\s$ holds as well. (3) follows from (2).
\end{proof}

 Observe that $\C_U$ is preserved by the inverse
operator, i.e. $f^{-1}\in\C_U$ whenever $f\in\C_U$. Indeed, we can
take $\SS_U(f^{-1})=\SS_U(f)$.  Observe as well that if
 $f$ is constant-like  over $U$ and $V\s U$ then $f$ is
 constant-like over $V$.

It is important and easy to see that \begin{equation}
\f(\C_U)=\C_{\psi(U)},\quad\f(\C_M)=\C_N\quad\hbox{and}\quad
\Phi(\SC_x)=\SC_{\tau(x)}.
\end{equation}
However, these facts require that $\SS_U(f)\leq S_U^M$ in
Definition 3.7. If we required $\SS_U(f)\leq F_U^M$ then (3.4)
would not hold, since we do not know whether $F_U^M$ is preserved
by $\f$ yet.

For $U\in\rob(B_M)$ we put $\C^{\id}_U:=\{f\in\H_G(M)_0: \,
f=\id\,\hbox{on}\,\pi^{-1}(U)\}$. Then we have \begin{equation}
\C^{\id}_U\s\C_U,\quad
\C_U\cdot\C^{\id}_U=\C_U^{\id}\cdot\C_U=\C_U\quad\hbox{and}\quad
\C_U^{\id}\cap\C_M=\{\id\}. \end{equation} Indeed, the first
follows from the fact that for $f\in\C_U^{\id}$ we may take
$\SS_U(f)=F_U^M$. The second is trivial, and the third follows
from the obvious fact: \begin{equation}
f\in\C_U\,\hbox{and}\,(\exists u\in \pi^{-1}(U)):\,
f(u)=u\quad\Rightarrow\quad f\in\C^{\id}_U, \end{equation} and the
same is true for $U$ replaced by $B_M$. Consequently, for all
$U,V\in\rob(B_M)$ we have:
\begin{equation}
V\s U\quad\Rightarrow\quad
\C_U\setminus\C^{\id}_U\s\C_V\setminus\C^{\id}_V.
\end{equation}

Next, for $x\in B_M$ we set
\begin{equation}
\SC^{\id}_x:=\{f\in\H_G(M)_0:\, (\exists U\in\rob(B_M)), \,
x\in\cl(U)\,\hbox{and}\, f\in\C_U^{\id}\}.
\end{equation}
Analogously as in (3.5) we get
\begin{equation}
\SC^{\id}_x\s\SC_x\quad\hbox{and}\quad
\SC_x^{\id}\cap\C_M=\{\id\}.
\end{equation}

\begin{lem}
\begin{enumerate}\item
Let $f\in\gau(M)$ and let $\U\s\rob(B_M)$ be a finite open
covering of $\pi(\supp(f))$. Then $f$
 can be written as $f=f_1\ldots f_r$, where
$\pi(\supp(\tilde f_i))\s U_i$ for some $U_i\in\U$, $i=1\ld r$.
\item Let $x\in B_M$. If $f\in\gau(M)$ such that $f(u)=u$ for some
(all) $u\in\pi^{-1}(x)$ then there are disjoint open sets $V_1,
V_2\in\rob(B_M)$ such that $x\in\cl(V_1)\cap\cl(V_2)$ and a
decomposition $f=f_1f_2$ with $f_i\in\gau(M)$ and
$f_i\in\C^{\id}_{V_i}$ for $i=1,2$.
\end{enumerate}
\end{lem}

\begin{proof}
(1)  Let $f\in\gau(M)$ and let $\pi(\supp(f))\s U_1\cup\ldots\cup
U_l$, where $U_i\in\rob(B_M)$. For any $i=1\ld l$ fix a chart
$\phi_i$ on $M$ over $U_i$. Let $\alpha_i:B_M\r\R$ be such that
$\supp(\alpha_i)\s U_i$, $i=1\ld l$, and $\sum\alpha_i=1$ on
$\pi(\supp(f))$. Fix as well a chart on $G$
$\chi:V\r\chi(V)\s\frak g$, where $V$ is an open neighborhood of
$e$ in $G$, $\frak g$ is the Lie algebra of $G$, and $\chi(V)$ is
an open neighborhood of 0 in $\frak g$. Now we can write $f=\bar
f_1\ldots\bar f_s$, where $\bar f_j$ are so small that for any
$i=1\ld l$ and for any $j=1\ld s$ the mapping $\bar
f_j^{(i)}:U_i\r V\s G$ is well-defined, c.f. (3.1). Define $f_1$
by $f_1^{(1)}=\chi^{-1}(\alpha_1\cd(\chi\bar f_1)^{(1)})$ on
$\pi^{-1}(U_1)$ and $f_1=\id$ off $\pi^{-1}(U_1)$. Now
$\pi(\supp(f_1^{-1}\bar f_1))$ is in $U_2\cup\ldots\cup U_l$
Continuing this procedure we get $\bar f_1=f_1\ldots f_{l}$, and
finally the required decomposition $f=f_1\ldots f_r$.

(2) Choose $U_i\in\rob_x(B_M)$ and let $f\in\gau(M)$ satisfy the
assumption. Arguing as in (1) we may assume that
$f^{(i)}:U_i\r\frak g$. Then $f^{(i)}(x)=0$. If $\dim B_M=1$ the
proof is obvious. Suppose $n=\dim B_M>1$. Denote $y=(y_1\ld
y_n)\in\R^n$. We identify $U_i$ with the subset  $Q\s\R^n$ given
by $Q:=\{y\in\R^n:\, |y_1|<1, |y_2|<1\}$ such that $x$ identifies
with 0. Put $V_1:=\{y\in Q: |y_1|>|y_2|,\,y_1>0\}$ and
$V_2:=\{y\in Q: |y_1|>|y_2|,\,y_1<0\}$. Observe that any $z\in
Q\setminus(V_1\cup V_2\cup\{0\})$ is uniquely written in the form
$z=ty+(1-t)\bar y$, where $t\in[0,1]$, $y\in \p V_1$ and $\bar
y:=(-y_1,y_2\ld y_n)\in\p V_2$. We define $f_1$ by means of
$f^{(i)}_1$ as follows: $f^{(i)}_1=0$ on $\cl(V_1)$,
$f^{(i)}_1=f^{(i)}$ on $\cl(V_2)$, and
$f^{(i)}_1(z)=ty+(1-t)f^{(i)}(\bar y)$. Then $f_1$ and
$f_2:=f_1^{-1}f$ satisfy the claim.
\end{proof}

\begin{prop} Let us
denote $\C^{\ZZ}_U:=\ZZ^{\H_G(M)_0}(\gau(M))\cd\C^{\id}_U$. For
all $U\in\rob(B_M)$ we have $\Phi(\C^{\ZZ}_U)=\C^{\ZZ}_{\psi(U)}$.
\end{prop}
\begin{proof}
Take a chart $\phi_i$ with $\dom(\phi_i)=U$. Fix $W,
W_1\in\rob(B_M)$ with $\cl(W_1)\s W\s\cl(W)\s U$ and fix $\V\s
\rob(B_M)$ with $(B_M\setminus U)\s \bigcup\V\s(B_M\setminus
\cl(W))$. Fix as well $\G_{\V}=\{h_V:\, V\in\V\}\leq \H_G(M)_0$
such that $\tilde h_V(U)=V$ for any $V\in\V$, c.f., e.g.,
\cite{hir}. Then for each $c\in\C_U$ we have
\begin{equation}
c\in\C^{\ZZ}_U\quad\Longleftrightarrow\quad \gau(M)\quad\hbox{is
generated by}\quad \bigcup_{h\in\G_{\V}}h\ZZ^{\gau(M)}(c)h^{-1}.
\end{equation}
Indeed, we have  that for any $c\in\C^{\ZZ}_U$ we have
$\gau_U(M)\s\ZZ^{\gau(M)}(c)$. Hence $(\Rightarrow)$ in view of
Lemma 3.8(1) and the equality
$h^{-1}\gau_V(M)h=\gau_{h^{-1}(V)}(M)$. To show $(\Leftarrow)$,
let $c\not\in\C^{\ZZ}_U$. Let $x\in W_1$. We may assume that
$c^{(i)}(x)=g\in G\setminus\ZZ(G)$. Otherwise,
$c=c(g).(c(g)^{-1}.c)\in\C^{\ZZ}_U$, due to Proposition 3.7. Take
$g_1\in G$ such that $g_1g\neq gg_1$. By a similar reasoning to
that in Lemma 3.8 there is $f\in\gau_W(M)$ such that $f=c_i(g_1)$
on $\pi^{-1}(W_1)$. Then $f\not\in\ZZ^{\gau(M)}(c)$ and, by
construction, $f$ cannot be a product of elements of
$\bigcup_{h\in\G_{\V}}h\ZZ^{\gau(M)}(c)h^{-1}$.

It follows from (3.10) and Theorem 3.4 that
$\f(\C^{\ZZ}_U)\s\C^{\ZZ}_{\psi(U)}$, as claimed.
\end{proof}

For $U\in B_M$ denote
$\hat\C^{\id}_U=\{f\in\C^{\id}_U:\;\supp(f)\s\pi^{-1}(B_M\setminus
\cl(U))\}$. Then $\hat\C^{\id}_U\lhd \C^{\id}_U$. The following
result due to the second-named author \cite{Ryb3} will be useful.

\begin{thm} For all $U\in\ro(B_M)$,
$\hat\C_U^{\id}$ is a perfect group, i.e. $\hat\C_U^{\id}$ is
equal to its own commutator subgroup
$[\hat\C_U^{\id},\hat\C_U^{\id}]$.
\end{thm}

In fact, it is a special case of Theorem 1.1 in \cite{Ryb3}.

\begin{prop}
For all $U\in\rob(B_M)$ we have
$\Phi(\C^{\id}_U)=\C^{\id}_{\psi(U)}$. Consequently,  for any
$x\in B_M$, $\Phi(\SC^{\id}_x)=\SC^{\id}_{\tau(x)}$.
\end{prop}
\begin{proof}
Let us denote $G_1=\C^{\ZZ}_U$,  $H_1=\ZZ^{\H_G(M)_0}(\gau(M))$
and $K_1=\C^{\id}_U$, and analogously $G_2$, $H_2$ and $K_2$ for
$B_N$ and $\psi(U)$. The following situation arises as a result of
Proposition 3.9. Let $\f:G_1\cong G_2$ be a group isomorphism.
Next let $H_i\lhd G_i$ and $K_i\lhd G_i$, $i=1,2$, be such that
$H_i\leq\ZZ(G_i)$, $G_i=H_iK_i$ and $H_i\cap K_i=\{e\}$. Moreover,
we have that $\f(H_1)=H_2$. It follows that any $g\in G_i$ is
written uniquely as $g=hk$ with $h\in H_i$ and $k\in K_i$ for
$i=1,2$. We have to show that $\f(K_1)=K_2$.

For any $k_1\in K_1$ we can write uniquely
$\f(k_1)=\chi(k_1)\psi(k_1)$. That is, for $g_1=h_1k_1$ we have
$\f(g_1)=\f(h_1)\chi(k_1)\psi(k_1)$. Here $\psi:K_1\cong K_2$ is
an isomorphism and $\chi:K_1\r H_2$ is a homomorphism, and both
are uniquely determined by $\f$. Clearly, $\f(K_1)\leq K_2$ iff
$\chi$ is the trivial homomorphism.

Choose arbitrarily $V\in\rob(B_M)$ such that $\cl(V)\s U$. In view
of Theorem 3.10 we have
$\C^{\id}_U\leq[\hat\C^{\id}_V,\hat\C^{\id}_V]=\hat\C^{\id}_V\leq\C^{\id}_V$.
It follows that $\f(\C^{\id}_U)\leq\C^{\id}_{\psi(V)}$. By
Proposition 3.9, $\f(\C^{\ZZ}_U)\leq \C^{\ZZ}_{\psi(U)}$. Since
$\C^{\id}_{\psi(V)}\cap\C^{\ZZ}_{\psi(U)}\leq\C^{\id}_{\psi(U)}$,
the first assertion follows. The second assertion holds in view of
the first.
\end{proof}

The symbol $F_x^M$ stands for the subgroup of $\gau(M)$ such that
$f\in F_x^M$ iff $f(u)=u$ for any (or some) $u\in\pi^{-1}(x)$. In
the proof of the following clue ingredient of the proof of Theorem
1.2 we are able to provide a condition which characterizes the
subgroups $F_x^M$ among other subgroups in $\gau(M)$ and which is
preserved by $\f$.

\begin{lem}
Under the assumptions of Theorem 1.2, let  $x\in B_M$. Then
$\f(F^M_x)=F^N_{\tau(x)}$.
\end{lem}

\begin{proof}

 Fix $U\in\rob_x(B_M)$.
We define a family of subgroups of $\gau(M)$ as follows:
$$
\F_x^M:=\{F\leq\gau(M):\, (1), (2), (3)\,\hbox{ and }\,(4)\,\hbox{
below hold }\},$$ where
\begin{enumerate}
\item $S_x^M\leq \N^{\H_G(M)_0}(F)$; \item for any
$h\in\gau(M)\setminus F$, there are $f\in F$ and
$c\in\C_U\setminus\C^{\id}_U$ such that $h=cf$; \item
$(\SC_x\setminus\SC_x^{\id})\cap F =\emptyset$ (c.f. Def.3.6(3),
(3.8) and (3.9)); \item $F$ is a minimal subgroup of $\gau(M)$
satisfying (1), (2) and (3).
\end{enumerate}

 Clearly $F_x^M$ satisfies (1). It satisfies also (2) in view of
 Proposition 3.7(1). Finally, $F_x^M$ fulfills (3) by (3.6).

Now, assume that $F\in\F_x^M$. Let $f\in F^M_x$. Then, due to
Lemma 3.8(2), there are disjoint open sets $V_1, V_2\in\rob(B_M)$
such that $x\in\cl(V_1)\cap\cl(V_2)$ and a decomposition
$f=f_1f_2$ with $f_i\in\gau(M)$ and $f_i\in\C^{\id}_{V_i}$ for
$i=1,2$. If, e.g., $f_1\not\in F$ then by (2) and (3.7) there
exist $\bar f\in F$ and $c\in\C_{V_1}\setminus\C^{\id}_{V_1}$ with
$f_1=c\bar f$. This contradicts (3), since $c^{-1}.f_1=\bar
f\in(\SC_x\setminus\SC_x^{\id})\cap F$. Therefore $f_1,f_2\in F$
and $f\in F$. Thus, if $F$ satisfies (2) and (3) then $F_x^M\leq
F$. It follows that $\F_x^M=\{F_x^M\}$. In view of (1), (3.4),
Proposition 3.11 and Theorem 3.4 we have that
$\f(F_x^M)=F_{\tau(x)}^N$.
\end{proof}

For $x\in B_M$  denote $\gau_x(M):=\{h|_{\pi^{-1}(x)}:\,h\in
\gau(M)\}\leq\H(\pi^{-1}(x))$. Notice that $\gau_x(M)\cong G$, see
(3.1), however an isomorphism is not canonical. There is neither a
canonical isomorphism $\pi^{-1}(x)\cong\gau_x(M)$, since otherwise
$\pi^{-1}(x)$ would admit a canonical group structure. However as
a consequence of Lemma 3.12 we have
\begin{cor}
 For any $x\in B_M$, there is a
well-defined group isomorphism $\f(x):\gau_x(M)\cong
\gau_{\tau(x)}(N)$ induced by $\f$.
\end{cor}

\section{Proof of Theorem 1.2 and further remarks}
Theorem 1.2(1) coincides with Theorem 3.4.

To prove the claims (2) and (3) we adopt the notation from section
3. Let $x\in U_i$ and let $\phi_i:\pi^{-1}(U_i)\r U_i\t G$ be a
local trivialization. Then we have the $G$-equivariant
identification $\phi_i^x:\pi^{-1}(x)\cong G$. On the other hand,
we have the canonical isomorphism $\kappa_G:G\ni
g\mapsto\mu_g\in\gau(G)$. It is easily checked that for $g\in G$
$\kappa_G(g)$ under the identification $\phi_i^x$ writes as
$$ \bar\phi_i^x(g):\pi^{-1}(x)\ni u\mapsto
(\phi^x_i)^{-1}(g).\phi^x_i(u)\in\pi^{-1}(x).$$ It follows that
given another local trivialization  $\phi_j:\pi^{-1}(U_j)\r U_j\t
G$ with $x\in U_j$ we get $\bar\phi^x_j(g)=\bar\phi^x_i(g)$. In
fact, we have $\phi_j^x=\mu_{g_{ji}}\ci\phi^x_i$. Hence
$(\phi_j^x)^{-1}=(\phi^x_i)^{-1}\ci\mu_{g_{ij}}$ and we get
\begin{align*}\bar\phi^x_j(g)(u)&=(\phi^x_j)^{-1}(g).\phi^x_j(u)=(\phi^x_i)^{-1}(g).g_{ij}.\phi^x_j(u)\\
&=(\phi^x_i)^{-1}(g).\phi^x_i(u)=\bar\phi^x_i(g)(u).\\
\end{align*}
We also have $\bar\phi^x_i(gh)=\bar\phi^x_i(g)\ci\bar\phi^x_i(h)$
for $g,h\in G$. Therefore the identification
$\phi_i^x:\pi^{-1}(x)\cong G$ can be applied to define the
isomorphism $$\bar\phi^x_i:G\cong \gau_x(M)$$ independently of
$i$. It follows the existence of a mapping $\bar\f:B_M\r \Aut(G)$
defined by means of $\f(x)$, $x\in B_M$, namely given by $$ \bar
\f(x)=(\bar\phi^{\tau(x)}_j)^{-1}\ci\f(x)\ci\bar\phi^x_i.$$
 Moreover,
for $U=U_i$ we obtain desired isomorphisms
$$\phi_U:\Sect(\pi^{-1}(U))\cong \gau(\pi^{-1}(U)),$$
$$\phi_{\tau(U)}:\Sect(\pi^{-1}(\tau(U)))\cong
\gau(\pi^{-1}(\tau(U))),$$ and a bijection
$\sigma_U:\pi^{-1}(U)\cong\pi^{-1}(\tau(U))$ (defined obviously by
$\bar\f(x)$, $x\in U$) such that the claim (3) holds true. Observe
that unless the principal $G$-bundle $\pi:M\r B_M$ is globally
trivial we cannot identify $\Sect(M)$ with $\gau(M )$.

It remains to show that $\bar\f:B_M\r\Aut(G)$ is continuous (the
situation is symmetric). It suffices to do this locally. Let $x\in
U\in\ro(B_M)$. Arguing by contradiction, let $x_n\r x$ in $U$,
$n=1,2\ldots$, be such that $\bar\f(x_n)\in\Aut(G)$ does not
converge to $\bar\f(x)$. Then there exist $K, V\s G$, where $K$ is
closed and $V$ is open, such that $\bar\f(x)\in \N(K,V)$ (i.e.
$\bar\f(x)(K)\s V$) and $\bar\f(x_{n_m})\not\in \N(K,V)$,
$m=1,2,\ldots$, c.f. \cite{bour}. Hence there are $k_m\in K$ with
$\bar\f(x_{n_m})(k_m)\not\in V$, $m=1,2,\ldots$. Put $k=\lim_mk_m$
passing possibly to a subsequence. By using a chart on $G$ at $k$
and applying Tietze's theorem we can find a continuous mapping
$f:U\r G$ such that $f(x)=k$ and $f(x_{n_m})=k_m$ for $m\geq m_0$.
Thus we get that $\ev\ci(\bar\f, f)(x_{n_m})\not\in V$ and
$\ev\ci(\bar\f, f)(x)\in V$ where $\ev$ is the evaluation map.
This contradicts the fact that $\bar\f\ci f=\f(f)\ci\tau$ is
continuous.

This completes the proof of Theorem 1.2.

\begin{rem}
(1) Consider the "generic" case $B_M=\{x\}$, $B_N=\{x\}$ i.e.
$M=N=G$. Then $\H_G(M)\cong G$ canonically and $\f:G\cong G$ is an
automorphism. The "resulting" $\sigma$ is equal to $\f$ and is not
an equivariant mapping. Also the condition $\f (h)=\sigma \circ h
\circ \sigma^{-1}$ from Theorem 1.1 cannot occur. That is, the
reconstruction in the category of $G$-spaces is a specific one and
it is not completely analogous to that in the topological
category.

(2) Let $M=B_M\t G$ and $N=B_N\t G$. Given $\tau:B_M\cong B_N$ and
$\bar\f:B_M\r\Aut(G)$ there is an isomorphism
$\f:\H_G(M)_0\cong\H_G(N)_0$ induced by $\tau$ and $\bar\f$ in
such a way that (1) and (2) in Corollary 1.3 are satisfied.
\end{rem}

Let us formulate another version of Theorem 1.2 with analogous
proof (see Theorem 2.5).

\begin{thm} Let $G$ be a compact  Lie group
acting freely  on paracompact connected manifolds $M$ and $N$.
Suppose that
\begin{enumerate}
\item $H(M) \leq\H_G(M)$ and $H(N)\leq\H_G(N)$ are transversely
factorizable, transversely transitive (c.f. Def. 2.4);  \item
$\ker \tilde{\pi}=\gau(M)$, where $\tilde\pi:H(M)\r\H(B_M)$ is
induced by $\pi$, and the same is true for $H(N)$; and \item for
any $U\in\ro(M)$ the group $\hat H_U(M):=\{h\in
H(M):\,\supp(h)\s\pi^{-1}(U)\}$ is a perfect group.
\end{enumerate}
 If there is a group isomorphism $\f:H(M)\cong H(N)$ then the claims analogous to (1), (2) and (3) in
 Theorem 1.2 are fulfilled. Moreover, if $M$ and $N$ are globally trivial then also
 assertions analogous to those from Corollary 1.3 hold.
\end{thm}

\section{The case of $G$-manifolds with one orbit type}

Recall basic facts on the $G$-spaces with one orbit type (see,
Bredon \cite{br}, section II, 5). Let $G$ a compact Lie group and
let $X$ be a $T_{3\frac{1}{2}}$ $G$-space with one orbit type
$G/H$ (that is, all isotropy subgroups are conjugated to $H$). Set
$N=\N^G(H)$ and $X^H=\{x\in X: h.x=x\,,\forall h\in H\}$. Then we
have the homeomorphism $G\t_NX^H\ni[g,x]\mapsto g(x)\in X$. That
is, the total space of the bundle over $G/N$ with the standard
fiber $X^H$ associated to the principal $N$-bundle $G\r G/N$ is
$G$-equivalent to $X$. In particular, the inclusion $X^H\subset X$
induces a homeomorphism $X^H/N\cong X/G$.

Denote $K=N/H$. Given an arbitrary $G$-space $Y$, there is a
bijection $\kappa_{X,Y}$ between $G$-equivariant mappings $X\r Y$
and $K$-equivariant mappings $X^H\r Y^H$ such that
$\kappa_{X,Y}(f)=f|_{X^H}$.

Notice that $K$ acts freely on $X^H$ and the homeomorphism
$X^H/N\cong X/G$ induces the homeomorphism $X^H/K\cong X/G$. In
particular, we get the principal $K$-bundle $\pi_X:X^H\r X/G$,
where $\pi_X$ is the restriction to $X^H$ of the projection
$\pi:X\r X/G$.

Now our generalization of Theorem 1.2 takes the following form.
\begin{thm}
Let $G$ be a compact  Lie group acting   on paracompact connected
smooth manifolds $M$ and $N$ with the same orbit type $G/H$.
 If there is a  group isomorphism
$\f:\H_G(M)_0\cong \H_G(N)_0$ then we get an isomorphism
$\hat\f:\H_K(M^H)_0\cong\H_K(N^H)_0$ given by
$\hat\f=\kappa_{M,N}\ci\f\ci\kappa_{M,N}^{-1}$. Furthermore, the
assertions of Theorem 1.2 hold if we replace $\pi_M:M\r B_M$ and
$\pi_N:N\r B_N$ by $\pi_M:M^H\r M/G$ and $\pi_N:N^H\r N/G$, resp.
In particular, there exist a homeomorphism $\tau:M/G\cong N/G$ and
a continuous mapping $\bar\f:M/G\r \Aut(K)$, and they induce local
homeomorphisms between the spaces $M^H$ and $N^H$.
\end{thm}
In fact, if $f\in\H_G(M)_0$ then $f|_{M^H}\in\H_K(M^H)_0$.


\begin{thebibliography}{99}


\bibitem{ban-m}
A. Banyaga, \wyr{The structure of classical diffeomorphism
groups}, Mathematics and its Applications, 400, Kluwer Academic
Publishers Group, Dordrecht, 1997

\bibitem{ban1}
A. Banyaga, \emph{On isomorphic classical diffeomorphism groups},
II, J. Diff. Geometry 28(1988), 23-35.

\bibitem{be-ru}
E. Ben Ami and M. Rubin, \emph{On the reconstruction problem of
factorizable homeomorphism groups and foliated manifolds },
Topology Appl. 157(2010), 1664-1679.

\bibitem{bo-br}
J. E. Borzellino, V. Brunsden, \emph{Determination of the
topological structure of an orbifold by its group of orbifold
diffeomorphisms}, J. Lie Theory 13(2003), 311-327.

\bibitem{bour}
N. Bourbaki, \emph{Groupes et alg\'ebres de Lie}, ch. III,
\'El\'ements de math., fasc. XXXVII, Hermann, Paris 1972.


\bibitem{br}
 G. E. Bredon, \wyr{ Introduction to compact transformation
groups},  Academic Press, New York - London  (1972)


\bibitem{ed-ki}
R. D. Edwards, R. C. Kirby, \emph{Deformations of spaces of
imbeddings}, Ann. Math. 93 (1971), 63-88

\bibitem{fil}
R. P. Filipkiewicz, \emph{Isomorphisms between diffeomorphism
groups}, Ergod. Th. Dynam. Sys., 2(1982), 159-171.

\bibitem{gle}
A. M. Gleason, \emph{Spaces with a compact Lie group of
transformations}, Proc. Amer. Math. Soc. 1(1950), 35-43.

\bibitem{gr}
J. Grabowski, \emph{Isomorphisms and ideals of the Lie algebra of
vector fields}, Invent. Math. 50(1978), 13-33.

\bibitem{gr-po}
J. Grabowski, N. Poncin, \emph{Lie algebra characterization of
manifolds}, Cent. Eur. J. Math. 2(2004), 811-825.

\bibitem{hir}
M. W. Hirsch, \wyr{Differential Topology}, Graduate Texts in
Mathemetics 33, Springer 1976.


\bibitem{kmr}
A. Kowalik, I. Michalik, T. Rybicki, \emph{Reconstruction theorems
for two remarkable groups of diffeomorphisms}, Travaux
math\'ematiques, 18(2008), 77-86.

\bibitem{KM}
A. Kriegl, P. W. Michor, \wyr{The Convenient Setting of Global
Analysis}, Mathematical Surveys and Monographs, vol.53, American
Mathematical Society, 1997.




\bibitem{mich}
P. W. Michor, \emph{Gauge theory for fiber bundles}, Bibliopolis,
Monographs and textbooks in physical sciences, Napoli 1991.

\bibitem{om}
H.Omori, \emph{Infinite dimensional Lie transformation groups},
Lecture Notes in Math. {\bf 427}, Springer-Verlag 1974.


\bibitem{Ru1}
M. Rubin, \emph{On the reconstruction of~topological spaces from
their groups of~homeomorphisms}, Trans.Amer.Math. Soc., {\bf 312}
(1989), 487-537.

\bibitem{Ru2}
M. Rubin, \emph{The reconstruction of trees from their
automorphism groups}, Contemporary Math., 151, Amer. Math. Soc.,
Providence RI, 1993.

\bibitem{Ru3}
M. Rubin, \emph{Locally moving groups and reconstruction problems
}, Ordered Groups and Infinite Permutation Groups Math. Appl. {\bf
354}, Kluwer Acad. Publ. Dordrecht, (1996), 121-151.

\bibitem{RY}
M. Rubin, Y. Yomdin, \emph{Reconstruction of manifolds and subsets
of normed spaces from subgroups of their homeomorphism groups},
Dissertationes Math. 435(2005), 1-246.

\bibitem{Ru4}
M. Rubin, \emph{A reconstruction theorem for smooth foliated
manifolds}, Geom. Dedicata 150(2011), 355-375.

\bibitem{Ryb1}
T. Rybicki, \emph{ Isomorphisms between groups of
diffeomorphisms}, Proc. Amer. Math. Soc. {\bf 123}(1995), 303-310.

\bibitem{Ryb4}
T. Rybicki, \emph{Isomorphisms between leaf preserving
diffeomorphism groups}, Soochow J. Math. 22(1996), 525-542.

\bibitem{Ryb5}
T. Rybicki, \emph{On automorphisms of a Jacobi manifold}, Univ.
Iagel. Acta Math. 38(2000), 89-98.

\bibitem{Ryb2}
T. Rybicki, \emph{ Isomorphisms between groups of homeomorphisms},
Geometriae Dedicata {\bf 93}(2001), 71-76.

\bibitem{Ryb3}
T. Rybicki, \emph{On commutators of equivariant homeomorphisms},
Topol. Appl. 154(2007), 1561-1564.

\bibitem{whit}
J. V. Whittaker, \emph{On isomorphic groups and homeomorphic
spaces}, Ann. Math. 78(1963), 74-91.


\end{thebibliography}
\end{document}